\documentclass{article}

\usepackage[english]{babel}

\usepackage[a4paper,top=2cm,bottom=2cm,left=3cm,right=3cm,marginparwidth=1.75cm]{geometry}

\usepackage{amsmath}
\usepackage{amssymb}
\usepackage{amsfonts}
\usepackage{amsthm}
\usepackage{graphicx}
\usepackage{xcolor}
\usepackage{tabularx}
\usepackage{tikz}

\newtheorem{theorem}{Theorem}[subsection]
\newtheorem{corollary}{Corollary}[theorem]
\newtheorem{lemma}[theorem]{Lemma}

\theoremstyle{definition}
\newtheorem{definition}[theorem]{Definition}
\newtheorem{Question}[theorem]{Question}

\title{A topological embedding of $T_3$ into $\mathbb{Z}^2$}
\author{Samuel Kelly}
\date{26 January 2024}

\begin{document}
\newcommand{\cut}[0]{\mathrm{cut}}
\newcommand{\cutwidth}[0]{\mathrm{cutwidth}}
\newcommand{\bandwidth}[0]{\mathrm{bandwidth}}
\newcommand{\sep}[0]{\mathrm{sep}}
\newcommand{\cw}[0]{\mathrm{cw}}
\newcommand{\bw}[0]{\mathrm{bw}}
\newcommand{\wir}[0]{\mathrm{wir}}
\newcommand{\vol}[0]{\mathrm{vol}}
\newcommand{\diam}[0]{\mathrm{diam}}
\newcommand{\conn}[0]{\mathrm{conn}}

\newcommand{\Ryan}[1]{\textcolor{red}{#1}}

\maketitle

\setcounter{tocdepth}{2}

\begin{abstract}
    We prove that for any finite tree $T$ with $n$ vertices and maximal degree $3$, there is a topological embedding of $T$ into the integer grid $Z^2$ which maps vertices to vertices and whose image meets at most $\frac{7}{3}n$ vertices. This recovers a weaker form of a result due to Valiant \cite{Valiant} with stronger constants. We address Question 7.7 of \cite{barrett2021thick}, giving the first example of a pair of graphs $X,Y$ such that there is no regular map $X\to Y$ but the coarse wiring profile of $X$ into $Y$ grows linearly.
\end{abstract}

\section{Introduction}
Coarse wiring profiles were introduced by Barratt-Hume as a coarse geometric variant of thick topological embeddings previously considered by Kolmogorov-Barzdin and Gromov-Guth \cite{barrett2021thick,Barzdin1993,GG-KB}. We first recall the key definitions from \cite{barrett2021thick}:

\begin{definition}[Wiring]\label{wiring def}
    Let $\Gamma, \Gamma'$ be graphs. A \textbf{wiring} of $\Gamma$ into $\Gamma'$ is a continuous map $f : \Gamma \to \Gamma'$ which maps vertices to vertices and edges to unions of edges. A wiring $f$ is a \textbf{coarse $k$-wiring} if
    \begin{enumerate}
        \item the restriction of $f$ to $V\Gamma$ is $\leq \text{$k$-to-$1$}$ i.e $|\{v \in V\Gamma | f(v) = w)| \leq k$ for all $w \in V\Gamma'$.
        \item each edge $e \in E\Gamma'$ is contained in at most $k$ of the image paths of $f$.
    
    \end{enumerate}
\end{definition}

\noindent The \textbf{volume} of a wiring $\vol(f)$ is the number of vertices in its image.

\begin{definition}[Wiring Profile]
Let $\Gamma$ be a finite graph and $Y$ be a graph. We denote by $\text{wir}^k(\Gamma \to Y)$ the minimal volume of a coarse $k$-wiring of $\Gamma$ into $Y$. If no such coarse $k$-wiring exists, we say $\text{wir}^k(\Gamma \to Y) = \infty$. Let $X$ and $Y$ be graphs. The $k$-\textbf{wiring profile} of $X$ into $Y$ is the function
\begin{equation*}
    \wir_{X \to Y}^k (n) = \max\Big\{ \wir^k(\Gamma \to Y) \mid \Gamma \leq X \text{ , } |\Gamma| \leq n  \Big\}
\end{equation*}
\end{definition}
\noindent Coarse wiring profiles are typically considered with respect to the following partial order. Given two functions $f,g:\mathbb{N}\to\mathbb{N}$, we write $f\lesssim g$ if there is a constant $C$ such that $f(n)\leq Cg(Cn)+C$ for all $n$, and we write $f\simeq g$ if $f\lesssim g$ and $g\lesssim f$. The goal of the paper is the following:

\begin{theorem}\label{thm:wireTZ2} Let $T_3$ denote the infinite $3$-regular tree, and let $\mathbb{Z}^2$ denote the 2-dimensional integer lattice. We have
\[
 \wir_{T_3 \to \mathbb{Z}^2}^k (n)\simeq n.
\]
\end{theorem}

\noindent To achieve this, we show that the $1$-wiring profile from $T_3$ to $\mathbb{Z}^2$ is linear. The framework for $1$-wirings into $\mathbb{Z}^2$ essentially matches the study of the theoretical model for VLSI. As such, a stronger result can be found in \cite{Valiant}, in which the author provides a bound on the area of the smallest rectangle containing the image graph with weaker constants. We obtain the following corollary by applying the permanence properties of wiring profiles proved in \cite{barrett2021thick}.

\begin{corollary} Let $Y$ be a Cayley graph of a finitely generated group $G$ which is not virtually cyclic. If, for all $k$,
\[
    \wir_{T_3 \to Y}^k (n) \not\simeq n
\]
then $G$ is amenable but not elementary amenable.
\end{corollary}
\begin{proof}
    When $G$ is virtually nilpotent but not virtually cyclic, $G$ admits a $\mathbb{Z}^2$ subgroup. To see this, note that $G$ admits a finite-index torsion-free subgroup $G'$ with non-trivial centre, let $x$ be such a non-trivial element. Since $G$ is not virtually cyclic, there is some element $y\in G'$ such that no power of $y$ is equal to a power of $x$. It follows that $x,y$ generate a $\mathbb{Z}^2$ subgroup of $G'$, and therefore of $G$.
    Therefore, $\wir_{\mathbb{Z}^2 \to Y}^k (n) \simeq n$ for all sufficiently large $k$. Applying Theorem \ref{thm:wireTZ2} and \cite[Lemma 3.2]{barrett2021thick}, we see that $\wir_{T_3 \to Y}^k (n) \simeq n$ for all sufficiently large $k$. \par

    \vspace{1em}

    When $G$ is elementary amenable but not virtually nilpotent, it admits a non-abelian free subsemigroup \cite{Chou}, so there is a regular map $T_3\to Y$, and $\wir_{T_3 \to Y}^k (n) \simeq n$ for all sufficiently large $k$. Alternatively, if $G$ is not amenable, then it admits a quasi-isometric embedding of $T_3$ \cite{Benjamini-Schramm}, so again $\wir_{T_3 \to Y}^k (n) \simeq n$ for all sufficiently large $k$.
\end{proof}

\noindent It is very natural to ask whether the following stronger result holds.

\begin{Question} Let $Y$ be a Cayley graph of a finitely generated group $G$. If, for some $k$,
\[
    \wir_{T_3 \to Y}^k (n) \simeq n
\]
is $G$ necessarily virtually cyclic?
\end{Question}
\section*{Acknowledgements}
The author was funded by a University of Bristol Student Summership Bursary throughout this work and would like to thank David Hume for his kind supervision. We would also like to thank Louis Esperet for bringing the reference \cite{Valiant} to our attention.

\section{The $1$-wiring profile $T_3 \to \mathbb{Z}^2$ is linear}

To prove Theorem \ref{thm:wireTZ2}, we provide a linear upper bound for the $1$-wiring profile via an algorithm whose input is a subgraph $\Gamma \subset T_3$ and whose output is a $1$-wiring of $\Gamma$ into $\mathbb{Z}^2$ of volume less than $\frac{7}{3} |\Gamma|$. A $1$-wiring $f: \Gamma \to \mathbb{Z}^2$ is injective on the vertex set and maps each edge to a unique path in $\mathbb{Z}^2$. The amount of vertices in the image of a path is equal to the length of the image path in the taxicab metric on $\mathbb{Z}^2$, thus
\[
 \vol(f) = 1 + \sum_{vw \in E\Gamma} d_{\mathbb{Z}^2} (f(v), f(w))
\]
\noindent The algorithm is defined recursively. We assign some vertex $v \in \Gamma$ as a root to obtain a parent-child hierarchy. At each stage in the algorithm, we suppose we have wirings for the subtrees stemming from the children. We place our root at $(0, 0)$ and copy over the wiring for the first child, placing its root on the $y$-axis, just high enough so that the entire wiring is in the upper half plane. If there is a second child, we again copy its wiring but this time we rotate it $90^\circ$ degrees clockwise before placing its root on the positive $x$-axis, just far away enough so that every image point has a positive $x$-coordinate. \par
\vspace{1em}

\noindent Formally, the construction is given by $S(\Gamma, v)$, which takes a binary tree $\Gamma$ and some root vertex $v \in \Gamma$ as parameters and outputs a wiring for $\Gamma$ in $\mathbb{Z}^2$ rooted at $v$. Denoting the children of $v$ by $v_1$ and $v_2$, the \textit{child subtree}, denoted $\Gamma_{v_i}$ of $v_i$ is the connected component of $\Gamma$ containing $v_i$ but not $v$. Given a wiring $f$ for $\Gamma$, we define the \textit{connector}, denoted $\conn(f)$, to be $\conn(f) = |\min \{y \mid (x, y) \in f(\Gamma)\} |$; this is the extra length needed to connect a child to its root. We define $S: (\Gamma, v) \mapsto f$ as follows, where $f$ denotes the wiring $f: \Gamma \to \mathbb{Z}^2$ we return. \par

\begin{enumerate}
    \item Set $f(v) = (0,0)$.
    \item If $v$ has children.

    \begin{enumerate}
        \item If $v$ has only one child $w$. Calculate the wiring for $\Gamma_{w}$ and denote it $f_w$. We map the $u \in \Gamma_{w} \subset \Gamma$ to $f(u) = f_w(u) + (0, \conn(f_w) + 1)$.
        \item If $v$ has two children $w_1, w_2$, order them so that $|T_{w_1}| \geq |T_{w_2}|$. Calculate their wirings and denote them $f_1$ and $f_2$ respectively. For the largest child, we do the same as in $(a)$. We map the $u \in \Gamma_{w_1} \subset \Gamma$ to $f(u) = f_{1}(u) + (0, \conn(f_1) + 1)$. For the second wiring, we first rotate the image vertices $90^\circ$ degrees clockwise about the origin, denoting this $f_2^*$. We then map the $u \in \Gamma_{w_2} \subset \Gamma$ to $f(u) = f_{2}^*(u) +( \conn(f_2) + 1, 0 )$.
    \end{enumerate}
    \item Return $f$.
\end{enumerate}

\noindent To see that $f$ is a $1$-wiring, observe that at each step, the first child lies in the North-East quadrant and the second child lies in the South-East quadrant. Let $B_n$ be the $n$-th rooted perfect binary tree with an extra degree $1$ vertex placed above the traditional root. We have illustrated the image of $B_4$ in Figure \ref{fig:T_n Wiring}. Note also that $B_3$'s image is illustrated in blue and $B_2$'s is illustrated in red (after being rotated clockwise $90^\circ$). The colouring scheme is to supplement a future argument. \par

\begin{center}
    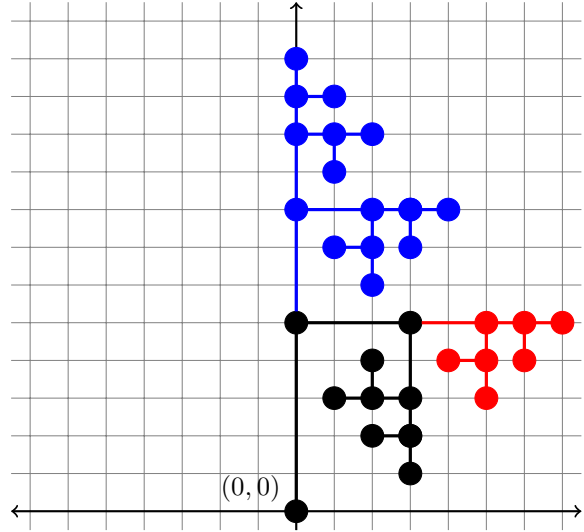
\begin{figure}[ht]
        \centering
        \begin{tikzpicture}
            \draw[step=0.5cm,gray,very thin] (-3.75,-2.75) grid (3.75,4.25);
            \draw[thick,<-] (-3.75,-2.5) -- (0,-2.5) ;
            \draw[thick,->] (0,-2.5) -- (3.75,-2.5) ;
            \draw[thick,->] (0,-2.75) -- (0,4.25);
            \filldraw (0, -2.5) circle [radius=.15cm];
            \filldraw (0, 0) circle [radius=.15cm];

            \node at (-0.6, -2.2) {$(0, 0)$};

            \filldraw[fill=blue, draw=blue] (0,1.5) circle [radius=.15cm];
            
            \filldraw[fill=blue, draw=blue] (0,2.5) circle [radius=.15cm];
            \filldraw[fill=blue, draw=blue] (1,1.5) circle [radius=.15cm];

            \filldraw[fill=blue, draw=blue] (0,3) circle [radius=.15cm];
            \filldraw[fill=blue, draw=blue] (0.5,2.5) circle [radius=.15cm];
            \filldraw[fill=blue, draw=blue] (1.5,1.5) circle [radius=.15cm];
            \filldraw[fill=blue, draw=blue] (0.5,2) circle [radius=.15cm];
            \filldraw[fill=blue, draw=blue] (1,1) circle [radius=.15cm];

            \filldraw[fill=blue, draw=blue] (0,3.5) circle [radius=.15cm];
            \filldraw[fill=blue, draw=blue] (0.5,3) circle [radius=.15cm];
            \filldraw[fill=blue, draw=blue] (1,2.5) circle [radius=.15cm];
            \filldraw[fill=blue, draw=blue] (2,1.5) circle [radius=.15cm];
            \filldraw[fill=blue, draw=blue] (1.5,1.5) circle [radius=.15cm];
            \filldraw[fill=blue, draw=blue] (1.5,1) circle [radius=.15cm];
            \filldraw[fill=blue, draw=blue] (0.5,1) circle [radius=.15cm];
            \filldraw[fill=blue, draw=blue] (1, 0.5) circle [radius=.15cm];

            \filldraw[fill=black, draw=black] (1.5,0) circle [radius=.15cm];
            \filldraw (1.5,-1) circle [radius=.15cm];
            \filldraw (1.5,-1.5) circle [radius=.15cm];
            \filldraw (1.0,-1.0) circle [radius=.15cm];
            \filldraw (1.5,-1.5) circle [radius=.15cm];
            \filldraw (1, -0.5) circle [radius=.15cm];
            \filldraw (0.5,-1) circle [radius=.15cm];
            \filldraw (1,-1.5) circle [radius=.15cm];
            \filldraw (1.5,-2) circle [radius=.15cm];

            \filldraw[fill=red, draw=red] (2.5,0) circle [radius=.15cm];
            \filldraw[fill=red, draw=red] (3,0) circle [radius=.15cm];
            \filldraw[fill=red, draw=red] (3.5,0) circle [radius=.15cm];
            \filldraw[fill=red, draw=red] (3,-0.5) circle [radius=.15cm];
            \filldraw[fill=red, draw=red] (2.5,-0.5) circle [radius=.15cm];
            \filldraw[fill=red, draw=red] (2,-0.5) circle [radius=.15cm];
            \filldraw[fill=red, draw=red] (2.5,-1) circle [radius=.15cm];

            \draw[black, very thick] (0,-2.5) -- (0,0);
            \draw[black, very thick] (0, 0) -- (1.5 ,0);
            \draw[black, very thick] (1.5, 0) -- (1.5,-2);
            \draw[black, very thick] (0.5,-1) -- (1.5,-1);
            \draw[black, very thick] (1, -1) -- (1,-0.5);
            \draw[black, very thick] (1,-1.5) -- (1.5,-1.5);

            \draw[blue, very thick] (0,0.15) -- (0,3.5);
            \draw[blue, very thick] (0,3) -- (0.5,3);
            \draw[blue, very thick] (0,2.5) -- (1,2.5);
            \draw[blue, very thick] (0.5,2.5) -- (0.5,2);
            \draw[blue, very thick] (0,1.5) -- (2,1.5);
            \draw[blue, very thick] (1.5,1.5) -- (1.5,1);
            \draw[blue, very thick] (1,1.5) -- (1,0.5);
            \draw[blue, very thick] (0.5,1) -- (1,1);

            \draw[red, very thick] (1.65,0) -- (3.5,0);
            \draw[red, very thick] (3,0) -- (3,-0.5);
            \draw[red, very thick] (2.5,0) -- (2.5,-1);
            \draw[red, very thick] (2,-0.5) -- (2.5,-0.5);
            
        \end{tikzpicture}
        \caption{The image of $B_4$}
        \label{fig:T_n Wiring}
    \end{figure}
\end{center}

\noindent We now turn our attention towards bounding the volume of this wiring. To motivate our approach, consider subdividing any edge in $B_4$ and the resultant change in volume of Figure \ref{fig:T_n Wiring}. The more rotations the edge we subdivide has undergone, the greater the volume added per subdivision. However, add too many and one changes the size ordering which results in the algorithm outputting a new wiring of lesser volume. Thus, maximising the volume of our tree is tantamount to maximising the amount of subdivisions on rotated edges whist adhering to the constraints of the size ordering. To prove Theorem \ref{thm:wireTZ2}, we simplify the problem by considering only one such size-ordering at a time and study how we can maximise the volume by distributing subdivisions according to this ordering. As an example, to subdivide the black (only) edge that has been rotated $4$ times and keep the same size ordering, we must apply at least $15$ other subdivisions.  For the following argument, we suppose for any graph that the root vertex $v$ on which we initiated the algorithm was of order $1$. We use this root to give us the parent-child hierarchy. We write $S(\Gamma, v) = f_\Gamma$. \par

\vspace{1em}

\noindent Let's standardise the meaning of our "size-ordering". Given our rooted tree $\Gamma$, we assign a left-to-right order for each pair of children by size so that the left child's subtree has more vertices than the right's. We define the \textit{reduction} of $\Gamma$, denoted $R(\Gamma)$, to be the series-reduced tree homeomorphic to $\Gamma$ with the same parent-child hierarchy and left-right ordering as $\Gamma$ imposed on its vertices. Note that it is the right children that get rotated. The behaviour of the volume of our wiring $S(\Gamma, v) $ is entirely determined by the reduction of $\Gamma$. We write $\Gamma_1 \sim \Gamma_2$ if $R(\Gamma_1) = R(\Gamma_2)$. A \textit{reduction} $R$ is the result of reducing a tree; it is a finite, ordered, rooted, full binary tree. For a reduction $R$, let $V_R$ denote the maximal volume of a graph $\Gamma \sim R$.
\[
V_R = \sup \bigg\{  \frac{\vol( f_\Gamma)}{|\Gamma|} \text{        }\bigg\rvert \text{   
     }\Gamma \sim R \bigg\}
\]
This $V_R$ is attained as a limit of graphs that are subdivisions of $R$ adhering to the ordering. To prove Theorem \ref{thm:wireTZ2}, we show that for every reduction $R$, $V_R \leq \frac{7}{3}$. Fortunately, we need only show this is true for the $R = B_n$.

\begin{lemma}
    Let $R$ be a reduction, there exists $n$ such that $V_R \leq V_{B_n}$.
\end{lemma}
\begin{proof}
    We first show a stronger result. Suppose that a reduction $R_1$ is a subtree of another reduction $R_2$ and that the parent-child hierarchy and left-right orderings for their vertices in common are equivalent, then $V_{R_1} \leq V_{R_2}$. \par

    \vspace{1em}
    
    \noindent Let $(\Gamma_m)_{m=1}^{\infty}$ be a sequence of graphs with $\lim_{m \to \infty} |\Gamma_m| =\infty$ that attains the supremum $V_{R_1}$, i.e. $\Gamma_m \sim R_1$ for all $m$, and
    \[
        \lim_{m\to\infty} \frac{f_{\Gamma_m}}{|\Gamma_m|} = V_{R_1}
    \]
    We now create a new sequence $(\Gamma'_m)_{m=1}^{\infty}$ by taking the vertices in $R_2 \setminus R_1$ and the graphs $\Gamma_m$, and appending the $R_2 \setminus R_1$ onto the leaves of each $\Gamma_m$ so that $\Gamma'_m \sim R_2$. We never subdivide any edges between the elements of $R_2 \setminus R_1$. Thus,
    \[
        \lim_{m\to\infty} \frac{f_{\Gamma'_m} + |R_2 \setminus R_1| }{|\Gamma_m| + |R_2 \setminus R_1| } = \lim_{m\to\infty} \frac{f_{\Gamma_m}}{|\Gamma_m|} = V_{R_1}
    \]
    and thus $V_{R_1} \leq V_{R_2}$. 

    \vspace{1em}
    
    \noindent Let $L(R)$ denote the maximal quantity of two-child parents in a path from root to leaf in $R$. The lemma follows as all reductions $R$ with $L(R) = n$ are subgraphs of $B_n$. 
\end{proof}
\noindent The remainder of the paper is dedicated to showing that $V(n) = V_{B_n}$ is uniformly bounded. Firstly, let's define a method for calculating $V_R$. 

\begin{lemma}
Enumerating the leaves of a reduction $R$ from left to right in order by $1, \dots, k$, let $v_i$ denote the volume per subdivision of leaf $i$ and $s_i \in [0, 1]$ denote the proportion of subdivisions we apply to leaf $i$. Then
\[
    V_R = \max \{ v_1 s_1 + \dots + v_k s_k  \mid s_1 + \dots +s_k = 1 \text{ and the $s_i$ satisfy the size-ordering}\}
\]
\end{lemma}
\begin{proof}
    Let $R$ be a reduction. We first show that we need only subdivide the edges connecting to leaves. \par

    \vspace{1em}

    \noindent Let $(\Gamma_n)_{n=1}^{\infty}$ be a sequence of graphs with $\Gamma_n \sim R$ attaining $V_R$. Every $\Gamma_n$ is a subdivision of $R$ satisfying the left-right ordering. A subdivision of an edge close to the root gives less volume per vertex added than a subdivision of a deeper edge. To produce our desired sequence of graphs $\Gamma'_n$, simply move all of the subdivisions of non-leaf edges to their left-most leaf child. The volume of this new sequence is greater than or equal to the volume of the original and, thus, also attains the supremum.\par    

    \vspace{1em}

    \noindent Following this approach of only subdividing leaf-edges, after applying $N$ subdivisions to $R$ the volume over the order is given by
    \[
    \frac{v_1 (s_1 N) + \dots + v_k (s_k N) + \vol(f_R)}{ N + |R|}
    \]
    whose limit as $N$ grows large yields the desired result.
\end{proof}

\noindent To bound $V_{B_n}$, we exploit the self-similar nature of our image of $B_n$. Recall the illustration of $B_4$ in Figure \ref{fig:T_n Wiring}. We see that $B_n$ is comprised of $B_{n-1}$ (in blue), $B_{n-2}$ (in red) and a spiral shape (in black) which we call $S_n$. We define $S_n$ to be the reduction that outputs the structure of the black nodes in Figure \ref{fig:T_n Wiring}, but such that the blue and red sections are straight paths (i.e. single leaves with many subdivisions) of the same quantity of nodes as in $B_{n-1}$ and $B_{n-2}$ respectively. Once we have calculated the value of $V_{S_n}$, we obtain the value of $V_{B_n}$ by replacing the two straight paths with copies of $B_{n-1}$ and $B_{n-2}$. As the edges in the spiral have undergone the most rotations, we place as many subdivisions there as possible. This leaves a quarter of the subdivisions left for the red $B_{n-2}$ and a half for the blue $B_{n-1}$. We have, note we must remove the straight paths from $S_n$,
\[
V_{B_n} = \frac{1}{2}( V_{B_{n-1}} - 1) + \frac{1}{4} (V_{B_{n-2}}-1) + V_{S_n}
\]
We will shortly prove that $V_{S_n} \leq \frac{4}{3}$ for all $n$. Thus,
\[
    V(n) \leq \frac{7}{12} + \frac{1}{2}V(n-1) + \frac{1}{4}V(n-2)
\]
with initial conditions $V(0) = 1$ and $V(1) = 1$. This is soluble, $V_n$ is monotone increasing and bounded above by its limit of $\frac{7}{3}$. Thus $V(n) \leq \frac{7}{3}$ for all $n$. We finish with the proof that $S_n$ is bounded above by $\frac{4}{3}$. \par
\begin{lemma}
    For all $n$, $V_{S_n} \leq \frac{4}{3}$.
\end{lemma}
\begin{proof}
    We must distribute the subdivisions of the leaves to maximise the volume. The left-most two leaves (the straight paths in place of the blue and red sections) provide a volume per subdivision of $1$. Thus, we assign the minimum amount of subdivisions to them, i.e. $\frac{1}{2}$ and $\frac{1}{4}$. Enumerate all of the other leaves from left to right by the numbers $0, \dots 2^{n-2} - 1$. Again, let $s_l$ denote the proportion of subdivisions we assign to leaf $l$ and $v_l$ denote the volume per subdivision of leaf $l$. Let $B(x)$ denote the amount of $1$s in the binary expansion of $x$. The amount of rotations a leaf $l$ undergoes is given by $B(l) + 2$ and $v_l$ is bounded by $2 + \lfloor B(l)/2 \rfloor$. In practise, very few of the leaves attain this bound. This is due to what we call shadowing. In Figure $\ref{fig:T_n Wiring}$, see that if $s_2 > s_1$, then $v_1 = 1$ and vice versa. Thus, if we have assigned any subdivisions to leaf $1$, we gain nothing from assigning them to leaf $2$. However, this is not the case for leaf $3$. The most rotated leaf produces no shadows. Thus, for $S_4$, our maximum is attained by assigning $s_3 = s_2 =\frac{1}{4}$, $s_1 = 0$ and $s_0 = \frac{1}{2}$. This gives (recall the first two leaves), $V_{S_4} = (1)\frac{1}{2} + (1)\frac{1}{4} + (2s_0 + 2s_2 + 3s_3)\frac{1}{4} = \frac{21}{16}$. \par

    \vspace{1em}

    \noindent For the general case, recall the more rotations a leaf has undergone, the more volume per subdivision. The more shadowed leaves we still extend, the less the volume. Thus, at any given parent $v$, we assign half the possible subdivisions to its left-most leaf in the subtree $\Gamma_v$ and half of them to be distributed by its right child who does the same. Now, we have no shadows interfering with each other and the depth of all the leaves we assign to is maximised. The leaves we assigned subdivisions to were the leaves at positions $0$ and $p_j$ with $p_j = \sum_{i = n-2-j} ^ {n-3} 2^i$. The leaf $0$ has $v_0 = 2$ and $s_0 = \frac{1}{2}$. The leaves $p_j$ have $v_{p_j} = 2 + \lfloor B(p_j)/2 \rfloor = 2 + \lfloor j/2 \rfloor$. Assigning the maximum amount of subdivisions to the deepest vertex gives $s_0 = \frac{1}{2}$ and $s_{p_j} = \frac{1}{2^{j+1}}$ for $1 \leq j \leq n-3$ and $s_{p_{n-2}} = \frac{1}{2^{n-2}}$. Thus,
    \begin{align*}
            V_{S_n} &= \frac{1}{2} + \frac{1}{4} + \frac{1}{4} \bigg( s_0 v_0 + s_{p_{n-2}}v_{p_{n-2}} + \sum_{j=1}^{n-3} s_{p_{j}}v_{p_{j}} \bigg) \\
            V_{S_n} &= 1 + \frac{2 + \lfloor(n-2)/2\rfloor}{2^{n}} + \sum_{j=1}^{n-3} \frac{2 + \lfloor j/2 \rfloor}{2^{j+3}}
    \end{align*}
which simplifies to
\[
V_{S_n} = \frac{4}{3} - \frac{1}{2^n}\bigg(2 + \left\lfloor\dfrac{n-1}{2}\right\rfloor - \left\lfloor\dfrac{n}{2}\right\rfloor  \bigg) \leq \frac{4}{3}
\]
\end{proof}

\end{document}